\documentclass[paper=a4, fontsize=11pt]{scrartcl} 

\usepackage[T1]{fontenc} 
\usepackage{fourier} 
\usepackage[english]{babel} 
\usepackage{amsmath,amsfonts,amsthm} 
\usepackage{amsmath, calligra, mathrsfs}
\usepackage{amssymb}
\usepackage{mathtools}
\usepackage{hyperref}
\usepackage{mathdots}
\usepackage{array}
\usepackage[mathscr]{euscript}
\usepackage{verbatim}
\usepackage[dvipsnames]{xcolor}
\usepackage{mdframed} 
\usepackage{soulutf8}
\usepackage{stmaryrd}
\usepackage{tikz-cd}
\usepackage{hyperref}
\usepackage{lipsum} 

\usepackage{sectsty} 
\allsectionsfont{\centering \normalfont\scshape} 

\usepackage{fancyhdr} 
\usepackage{xurl}
\newcommand*{\sheafhom}{\mathcal{H}\kern -.5pt om}
\numberwithin{equation}{section} 
\numberwithin{figure}{section} 
\numberwithin{table}{section} 

\newtheorem{thm}{Theorem}[section]
\newtheorem{cor}[thm]{Corollary}
\newtheorem{prop}[thm]{Proposition}

\theoremstyle{definition}

\theoremstyle{remark}
\newtheorem{rem}[thm]{Remark}

\DeclareMathOperator{\lk}{lk}

\DeclareMathOperator{\Susp}{Susp}

\DeclareMathOperator{\Cone}{Cone}

\DeclareMathOperator{\Ast}{ast}

\newcommand{\horrule}[1]{\rule{\linewidth}{#1}} 

\title{	
	\normalfont \normalsize 
	\textsc{} \\ [25pt] 
	\horrule{0.5pt} \\[0.4cm] 
	\huge Repeated Lefschetz-like decompositions for flag doubly Cohen--Macaulay simplicial complexes and gamma vectors of flag spheres
	
	\horrule{2pt} \\[0.5cm] 
}

\author{Soohyun Park \\ \href{mailto:soohyun.park@mail.huji.ac.il}{soohyun.park@mail.huji.ac.il} } 

\date{\normalsize November 14, 2024} 

\begin{document}
	
	\maketitle 
	
	\begin{abstract}
		\noindent We find decompositions of $h$-polynomials of flag doubly Cohen--Macaulay simplicial complex that yield a direct connection between gamma vectors of flag spheres and constructions used to build them geometrically. More specifically, they are determined by iterated double suspensions and a ``net nonnegative set of edge subdivisions'' taking it to the given flag doubly Cohen--Macaulay simplicial complex. By a ``net nonnegative set of edge subdivision'', we mean a collection of edge subdivisions and contractions where there are at least as many edge subdivisions as contractions. \\
		
		\noindent Returning to the flag spheres, these repeated decompositions involve links over collections of disjoint edges and give an analogue of a Lefschetz map that applies to each step of the decomposition. The constructions used also give a direct interpretation of the Boolean decompositions coming from links and those of the entire simplicial complex. Roughly speaking, the Boolean vs. non-Boolean distinction is used to measure how far a flag sphere is from being the boundary of a cross polytope. An analogue of this statement for flag doubly Cohen--Macaulay simplicial complexes would replace boundaries of cross polytopes by repeated suspensions of links over edges of the given simplicial complex.
	\end{abstract}

	
	\section*{Introduction}
	We find an interpretation of the gamma vector of a flag sphere $\Delta$ in terms of a repeated decomposition coming from a ``net'' nonnegative number of edge subdivisions connecting the double suspension $\Susp^2 \lk_\Delta(e)$ of its link $\lk_\Delta(e)$ over an edge $e \in \Delta$ to Boolean decompositions $\Gamma = \{ F \cup G : F \in S, G \in 2^{[d - 2|F|]} \}$ (the vertex set $[d]$ being disjoint from $S$) of simplicial complexes $\Gamma$ whose $f$-vectors are equal to $h$-vectors of $\Delta$ (Theorem \ref{booldecgeo}). By ``net nonnegative edge subdivisions'', we mean a sequence of edge subdivisions and contractions where there are at least as many subdivisions as contractions. An application to local and global decompositions over links $\lk_\Delta(e)$ and the entire simplicial complex $\Delta$ is given in Corollary \ref{locglobgeo}. \\
	
	We note that the main repeated decomposition is somewhat analogous to a Lefschetz map and that the decomposition applies to pure flag doubly Cohen--Macaulay simplicial complexes with $\dim \Ast_\Delta(p) = \dim \Delta$ in general (Proposition \ref{susplowerbd}). Along the way, this gives an explicit way of thinking about the Boolean part $2^{[d]}$ of $\Gamma$ in terms of the boundary of a cross polytope of dimension $d$ and the non-Boolean part as a collection repeated net nonnegative edge subdivisions giving the ``remainder'' taking $\Susp^2 \lk_{\Delta_i}(e_i)$ for various $\Delta_i$ and edges $e_i \in \Delta_i$ (Proposition \ref{vertsubdivcount}). Finally, we note that the same suspension vertices may play different roles in a Boolean decomposition for indices giving faces of different sizes although they can be counted as if they belong to a Boolean partition (Remark \ref{suspbool}). \\

	\section{Repeated double suspension-net nonnegative edge subdivision decompositions}
	
	We find an analogue of a Lefschetz map that directly leads to a geometric interpretation of the gamma vector coordinates and Boolean decompositions in terms of maps connecting PL homeomorphic flag simplicial complexes. \\
	
	\begin{prop} \label{susplowerbd}
		Let $\Delta$ be a pure Cohen--Macaulay simplicial complex such that $\Ast_\Delta(p) = \Delta \setminus p$ is also Cohen--Macaulay and $\dim \Ast_\Delta(p) = \dim \Delta$ for any vertex $p \in V(\Delta)$. The first condition is called being ``doubly Cohen--Macaulay'' (p. 21 of \cite{Athsom}). Then, we have that $h_\Delta(x) \ge (1 + x) h_{\lk_\Delta(p)}(x)$ if $\dim \Ast_\Delta(p) = \dim \Delta$ (e.g. if $\Delta$ a sphere). Since $\Delta$ is flag, this implies that $h_\Delta(x) \ge (1 + x)^{|F|} h_{\lk_\Delta(F)}(x)$ for any face $F \in \Delta$ such that $\dim \Ast_\Delta(F) = \dim \Delta$. For example, $h_\Delta(x) \ge (1 + x)^2 h_{\lk_\Delta(e)}(x)$ for any edge $e \in \Delta$ if $\dim \Ast_\Delta(e) = \dim \Delta$. \\ 
	\end{prop}
	
	\begin{proof}
		Let $d = \dim \Delta + 1$. By Lemma 4.1 on p. 27 of \cite{Athsom}, we have that $h_\Delta(x) = h_{\Ast_\Delta(v)}(x) + x h_{\lk_\Delta(v)}(x)$ if $\dim \Ast_\Delta(v) = \dim \Delta$. Note that $\lk_\Delta(v)$ is Cohen--Macaulay if $\Delta$ is Cohen--Macaulay and that $\Ast_\Delta(v)$ is also Cohen--Macaulay since $\Delta$ is doubly Cohen--Macaulay (by definition). This implies that $\lk_\Delta(v) \subset \Ast_\Delta(v)$ is an inclusion of Cohen--Macaulay simplicial complexes. Here, we have that $\dim \lk_\Delta(v) = d - 2$ and $\dim \Ast_\Delta(v) = d - 1$. Since $\Delta$ is flag, no set of $d - 1$ vertices of $\lk_\Delta(v)$ forms a face of $\Ast_\Delta(v)$. Since $\dim \lk_\Delta(v) = d - 2$, its largest faces have $d - 1$ vertices. Suppose there is a $d$-tuple $T$ of vertices of $\lk_\Delta(v)$ forming a face of $\Ast_\Delta(v)$. Since vertices of $\lk_\Delta(v)$ are not equal to $v$, this is equivalent to having a $d$-tuple $T$ of vertices of $\lk_\Delta(v)$ forming a face of $\Delta$. Since $\Delta$ is flag, this is equivalent to every pair of vertices in $T$ being in $\Delta$ and every vertex in $T$ being in $\lk_\Delta(v)$ (i.e. $(v, w) \in \Delta$ for each vertex $w \in V(T)$). Since $\Delta$ is flag, this means that $T \cup v \in \Delta$ since pairs of vertices $(p, q)$ with $p, q \in T$ are already in $\Delta$ and $(v, w) \in \Delta$ for all $w \in T$. However, this would contradict $\dim \Delta = d - 1$ since $|T \cup v| = d + 1$. \\
		
		Then, Theorem 9.1 on p. 126 -- 127 of \cite{St} implies that $h_i(\lk_\Delta(v)) \le h_i(\Ast_\Delta(v))$ for all $i$. In particular, we have that 
		
		\begin{align*}
			h_i(\Delta) &= h_i(\Ast_\Delta(v)) + h_{i - 1}(\lk_\Delta(v)) \\
			&\ge h_i(\lk_\Delta(v)) + h_{i - 1}(\lk_\Delta(v)) \\
			&= [x^i] (1 + x) h_{\lk_\Delta(v)}(x).
		\end{align*}
		
		Finally, the last statement follows from the fact that $\lk_{\lk_\Delta(p)}(q) = \lk_\Delta(p, q)$ if $(p, q) \in \Delta$ and $\lk_\Delta(P) \cap \lk_\Delta(Q) = \lk_\Delta(P \cup Q)$ if $P \cup Q \in \Delta$ if $\Delta$ is flag. \\
		
	\end{proof}

	\begin{prop} \label{vertsubdivcount} ~\\
		Recall that two PL homeomorphic simplicial complexes can be connected by a sequence of edge subdivisions and contractions sch that each of the simplicial complexes in the sequence are flag (Theorem 1.2 on p. 70 and 77 of \cite{LN}). We will take the ``net number of edge subdivisions'' to be the difference between the number of edge subdivisions and the number of edge contractions. In this context, we have the following: \\
		\begin{enumerate}
			\item The net number of edge subdivisions to get from a flag simplicial complex $\Delta_1$ to another flag simplicial complex $\Delta_2$ that is PL homeomorphic to it is given by $f_0(\Delta_2) - f_0(\Delta_1) = h_1(\Delta_2) - h_1(\Delta_1)$. \\
			
			\item Any flag sphere $\Delta$ of dimension $d - 1$ is obtained from the boundary $\partial C_d$ of a $d$-dimensional cross polytope by a net nonnegative number of edge subdivisions. As mentioned above, ``net nonnegative'' means that there are $r$ subdivisions and $s$ edge contractions with $r \ge s$ and the net number given by $r - s$. Equivalently, we have that $\gamma_1(\Delta) \ge 0$ for any flag sphere $\Delta$ since $h_1 = \gamma_1 + d$. \\ 
			
			\item  The statements in Part 2 still hold if we replace $\partial C_d$ by the double suspension $\Susp^2 \lk_\Delta(e)$ of $\lk_\Delta(e)$ for any edge $e \in \Delta$ or the suspension $\Susp \lk_\Delta(v)$ for any vertex $v \in V(\Delta)$. In other words, we can get to $\Delta$ from any of these using a net nonnegative number of edge subdivisions. \\

		\end{enumerate}
	\end{prop}
	
	\begin{proof}
		\begin{enumerate}
			\item When we take an edge subdivision $\Delta'$ of a simplicial complex $\Delta$, we have that $f_0(\Delta') = f_0(\Delta) + 1$. Going in the reverse direction, we have that $f_0(\widetilde{\Delta}) = f_0(\Delta) - 1$ for an edge contraction $\widetilde{\Delta}$ of $\Delta$ since two vertices of $\Delta$ are identified with each other. This means that the change in the number of vertices keeps track of the net number of edge subdivisions. The statement then follows from $h_1 = f_0 - d$ (see Corollary 5.1.9 on p. 213 of \cite{BH}). \\
			
			\item  This statement follows from combining Part 1 with a result of Athanasiadis (Theorem 1.3 on p. 19 and 27 of \cite{Athsom}) that $h_i(\Delta) \ge \binom{d}{i}$ for any double Cohen--Macaulay flag simplicial complex $\Delta$ of dimension $d - 1$ (e.g. a flag sphere of this dimension). \\
			
			\item  As in Part 2, this statement follows from combining Part 1 with Proposition \ref{susplowerbd}. \\ 
		\end{enumerate}
	\end{proof}

	\begin{cor}
		Starting with a square, we can get to a flag sphere of any dimension by a sequence of suspensions followed by a net nonnegative number of edge subdivisions. \\
	\end{cor}

	\begin{rem} \textbf{(Suspensions and Boolean decompositions) \\} \label{suspbool}
		While suspensions are compatible with Boolean decompositions, we note that the suspending vertices play different roles depending on the degree of the term considered. For example, consider a flag sphere $\Delta$ and its link $\lk_\Delta(e)$ over some edge $e \in \Delta$. Suppose that \[ \Gamma_e = \{ F_e \cup G_e : F_e \in S_e, G_e \in 2^{[d - 2|F_e| - 2]} \} \] has an $f$-vector which is equal to $h(\lk_\Delta(e))$. In a simplicial complex $\Gamma$ such that $f(\Gamma) = h(\Delta)$, the suspension of $\Delta$ on the right hand side has the role of coning by a vertex since we multiply by $1 + x$ on each side. However, the role of the coning new vertices to $\Gamma$ when we multiply by $(1 + x)^2$ depends on the degree of the term considered. In general, we have that \[ h_i(T) = \sum_{0 \le j \le i} \gamma_j \binom{D - 2j}{i - j} \] if $\dim T = D - 1$ (Observation 6.1 on p. 1377 of \cite{NPT}). If there is a Boolean decomposition, the second term comes from the Boolean part. Replacing $D$ by $D + 2$ would mean that the new suspension vertices play the role of $D - 2i + 1$ and $D - 2i + 2$, which depends on the index $i$ being considered although they come from the same vertices of $\Susp^2 T$ if we assume that $T$ itself has the same $h$-vector as the $f$-vector of a simplicial complex with a Boolean decomposition. \\
	\end{rem}

	While the change in $h_1(\Delta) = f_0(\Delta) - d$ (equivalently change in $f_0(\Delta)$) keeps track of the \emph{net} number of edge subdivisions and the \emph{total} number of new non-Boolean vertices, the $m$-tuples of new non-Boolean vertices are added when we have intersections $(\Gamma_0)_{e_0} \cap (\Gamma_1)_{e_1} \cap \cdots \cap (\Gamma_m)_{e_m}$ of $m$-tuples of subcomplexes $(\Gamma_i)_{e_i}$ for $1 \le i \le m$. On the flag sphere side, this seems to come from faces of $\lk_{\Delta_i}(e_i)$ that contain $e_j$ for $j > i$. Removing $e_j$, the ``interesting'' part seems to be from $\lk_{\Delta_i}(e_i) \cap \lk_{\Delta_i}(e_j) = \lk_{\Delta_i}(e_i \cup e_j)$ since $e_i \cup e_j \in \Delta$ in this case. Taking unions of relevant edges, we can see that the number of new non-Boolean vertices added is $\le \frac{d}{2}$. Going in the reverse direction (edge contractions on $\Delta_i$), we seem to take antistars $\Ast_{\Gamma_i}(u_i)$ for some non-Boolean vertex $u_i$ such that $\Gamma_{e_i} = \lk_{\Gamma_i}(e_i)$. \\

	\begin{rem} \textbf{(Comments on remainder terms) \\}
		Let $\Delta_{e'}$ be the (stellar) subdivision of $\Delta$ with respect to an edge $e' \in \Delta$. Then, we have \[ h_{\Delta_{e'}}(x) - h_\Delta(x) = x h_{\lk_\Delta(e')}(x). \] 
		
		Now fix an edge $e \in \Delta$ and consider the change in the ``normalized'' remainder term obtained after dividing $h_\Delta(x) - (1 + x)^2 h_{\lk_\Delta(e)}(x)$ by $x$. The main point is to look at the change in the link since \[ (h_{\Delta_{e'}}(x) - (1 + x)^2 h_{\lk_{\Delta_{e'}}(\widetilde{e})}(x)) - (h_\Delta(x) - (1 + x)^2 h_{\lk_\Delta(e)}(x)) = ((h_{\Delta_{e'}}(x) - h_\Delta(x)) - (1 + x)^2 (h_{\lk_{\Delta_{e'}}(\widetilde{e})}(x)) - h_{\lk_\Delta(e)}(x)). \]
		
		Note that $h_{\lk_{\Delta_{e'}}(\widetilde{e})}(x)) - h_{\lk_\Delta(e)}(x)$ itself is divisible by $x$ and can be interpreted in terms of repeated edge subdivisioins and contractions. \\
		
		We would like to compare the second difference with what we get from edge subdivisions of $(d - 3)$-dimensional flag spheres. \\

		If $\widetilde{e} \in \Delta$ (i.e. $v \notin \widetilde{e}$), it is an edge of $\Delta$. For example, suppose that $\widetilde{e} = e$ and $e \in \lk_\Delta(e')$ (equivalently $e' \in \lk_\Delta(e)$). Since $\Delta'$ is also flag and $e \in \lk_{\Delta}(e')$, we have
		
		\begin{align*}
			\lk_{\Delta_{e'}}(e) &= \lk_{\Delta_{e'}}(r, s) \\
			&= \lk_{\Delta_{e'}}(r) \cap \lk_{\Delta_{e'}}(s).
		\end{align*}
		
		Since the simplicial complexes we're intersecting are flag, they are determined by their edges. Without loss of generality, it suffices to look at $\lk_{\Delta_{e'}}(r)$. The edges of $\lk_{\Delta_{e'}}(r)$ are given by the edges of $\lk_\Delta(r)$ excluding $e$ itself, $(v, r)$, $(a, v)$, and $(v, b)$. This means that $\lk_{\Delta_{e'}}(r) = (\lk_\Delta(r))_{e'}$. Using the same reasoning for $\lk_{\Delta_{e'}}(s)$ and taking intersections, this implies that $\lk_{\Delta_{e'}}(e) = (\lk_\Delta(e))_{e'}$. \\
		
		In particular, this would mean that 
		
		\begin{align*}
			h_{\lk_{\Delta_{e'}}(\widetilde{e})}(x)) - h_{\lk_\Delta(e)}(x) &= h_{\lk_{\Delta_{e'}}(e)}(x)) - h_{\lk_\Delta(e)}(x) \\
			&= h_{(\lk_\Delta(e))_{e'}}(x) - h_{\lk_\Delta(e)}(x) \\
			&= x h_{\lk_\Delta(e \cup e')}(x) \\
			\Longrightarrow (h_{\Delta_{e'}}(x) - (1 + x)^2 h_{\lk_{\Delta_{e'}}(\widetilde{e})}(x)) - (h_\Delta(x) - (1 + x)^2 h_{\lk_\Delta(e)}(x)) &= ((h_{\Delta_{e'}}(x) - h_\Delta(x)) \\ 
			&- (1 + x)^2 (h_{\lk_{\Delta_{e'}}(\widetilde{e})}(x)) - h_{\lk_\Delta(e)}(x)) \\
			&= x h_{\lk_\Delta(e')}(x) - (1 + x)^2 x h_{\lk_\Delta(e \cup e')}(x) \\
			&= x (h_{\lk_\Delta(e')}(x) - (1 + x)^2 h_{\lk_\Delta(e \cup e')}(x)) \\
			&= x (h_{\lk_\Delta(e')}(x) - (1 + x)^2 h_{\lk_{\lk_\Delta(e')}(e)}(x)) \\
			\Longrightarrow \frac{1}{x} (h_{\Delta_{e'}}(x) - (1 + x)^2 h_{\lk_{\Delta_{e'}}(\widetilde{e})}(x)) - \frac{1}{x} (h_\Delta(x) - (1 + x)^2 h_{\lk_\Delta(e)}(x)) &= h_{\lk_\Delta(e')}(x) - (1 + x)^2 h_{\lk_{\lk_\Delta(e')}(e)}(x)
		\end{align*}
		
		if $\widetilde{e} \in \Delta$ (i.e. $v \notin \widetilde{e}$),  $\widetilde{e} = e$, and $e \in \lk_\Delta(e')$ (equivalently $e' \in \lk_\Delta(e)$). In other words, we obtain the remainder term for the approximation of the $h$-polynomial of $\lk_{\Delta}(e')$ by the $h$-polynomial of the double suspension of $\lk_{\lk_\Delta(e')}(e) = \lk_\Delta(e \cup e')$. \\
		
		Note that \emph{any} flag sphere of dimension $D$ can be expressed as the link over an edge of \emph{some} flag sphere of dimension $D$. To see this, consider the double suspension of the starting flag sphere and take the edge to be one connecting a suspension point from the first suspension with a suspension point from the second suspension. This means that all the possible double suspension approximation remainder terms from a $(d - 3)$-dimensional flag sphere are covered by the differences above. Finally, the nonnegativity of the remainder terms can be used to write an inequality relating the remainder term of the link to that of the entire flag sphere. \\
	\end{rem}

	\color{black}
	
	We can think about successive non-Boolean edges added to the same face of (a modification of) $\Gamma$ in terms of (net) nonnegative subdivisions used to decompose the $h$-polynomial $\lk_\Delta(F)$ for faces $F \in \Delta$ that can be expressed in terms of unions of (disjoint) edges. \\
	
	In particular, we have the following: \\
	
	\begin{thm} \label{booldecgeo}
		The coefficients of the $\gamma$-polynomial of the $h$-polynomial of a flag sphere $\Delta$ are induced by repeated decompositions of links in Proposition \ref{susplowerbd}. In particular, step $m$ of this decomposition results in sums of terms of the form \[ x^r (1 + x)^{2m - 2r} \left( \sum_{j = 1}^{M_m} h_{\lk_{\Delta_m}(\alpha_{1j} \cup \cdots \cup \alpha_{mj})}(x) - \sum_{\ell = 1}^{N_m} h_{\lk_{\Delta_m}(\beta_{1\ell} \cup \cdots \cup \beta_{m\ell})}(x) \right)  \] for collections of disjoint edges $\alpha_{kj} \in \Delta_m$ and $\beta_{r\ell} \in \Delta_m$, where $\Delta_m$ is obtained from $\Delta$ via edge subdivisions and contractions, $M_m \ge N_m$ and $0 \le r \le m$. Note that this implies that at most $\frac{d}{2}$ steps with nonzero terms when $\dim \Delta = d - 1$. The repeated decomposition is analogous to a formal version of a Lefschetz map (Theorem 3.1 on p. 14 -- 15 of \cite{Pbalfveclef} from Theorem 14.1.1 on p. 237 -- 238 of \cite{Ara}, Theorem on p. 88 of \cite{CMSP}, p. 122 of \cite{GH}, Theorem 6.3 on p. 138 of \cite{Voi}) after comparing the degree $2m - 2r$ of the Boolean-like part from taking suspensions and the degree $r$ of the initial non-Boolean part. Roughly speaking, this is measuring how far away we are from being the boundary of a cross polytope (or at least a suspension). \\ 
	\end{thm}
	
	\color{red}

	\color{black}
	
	\begin{proof}
		If $\Delta'$ is the subdivision of $\Delta$ with respect to an edge $e \in \Delta$, we have that 
		\begin{align*}
			h_{\Delta'}(x) &= h_\Delta(x) + x h_{\lk_\Delta(e)}(x).
		\end{align*}
		
		If $\lk_\Delta(e)$ itself is the (stellar) subdivision of a simplicial complex $T$ with respect to an edge $e_T \in T$, we have 
		\begin{align*}
			h_{\Delta'}(x) &= h_\Delta(x) + x h_{\lk_\Delta(e)}(x) \\
			&= h_\Delta(x) + x (h_T(x) + x h_{\lk_T(e_T)}(x)) \\
			&= h_\Delta(x) + x h_T (x) + x^2 h_{\lk_T(e_T)}(x).
		\end{align*}
		
		This gives an idea for which faces of $\Gamma$ such that $h(\Delta) = f(\Gamma)$ have multiple non-Boolean vertices added to them after (net nonnegative) edge subdivisions (e.g. from the boundary of a cross polytope $\partial C_d$ or the double suspension $\Susp^2 \lk_\Delta(e)$ of the link $\lk_\Delta(e)$). \\
		
		We will look at how the decomposition above looks for the examples mentioned above and what ``nesting'' with more vertices added looks like. \begin{align*}
			h_\Delta(x) &= (1 + x)^2 h_{\lk_\Delta(e)}(x) + x h_{\lk_{\Delta_0}(e_0)}(x) + x h_{\lk_{\Delta_1}(e_1)}(x) + \ldots + x h_{\lk_{\Delta_r}(e_r)}(x) \\
			&- x h_{\lk_{\Delta_{r + 1}}(e_{r + 1})}(x) - \ldots - x h_{\lk_{\Delta_{r + s}}(e_{r + s})}(x) \\
			&= (1 + x)^2 h_{\lk_\Delta(e)}(x) + x (h_{\lk_{\Delta_0}(e_0)}(x) +  h_{\lk_{\Delta_1}(e_1)}(x) + \ldots + h_{\lk_{\Delta_r}(e_r)}(x) \\
			&- h_{\lk_{\Delta_{r + 1}}(e_{r + 1})}(x) - \ldots - h_{\lk_{\Delta_{r + s}}(e_{r + s})}(x)) 
		\end{align*}
		
		for some $r \ge s$ and setting $\Delta_0 = \Delta$ with $\Delta_i$ for $i \ge 1$ denoting subsequent edge subdivisions. Note that the expression in brackets ends up having nonnegative coefficients. \\

		Now consider some $\lk_{\Delta_i}(e_i)$. If there is an edge $q_i \in \Delta_i$ such that $e_i \cup q_i \in \Delta_i$ and $e_i \cap q_i = \emptyset$, then we can write \[ h_{\lk_{\Delta_i}(e_i)}(x) = (1 + x)^2 h_{\lk_{\Delta_i}(e_i \cup q_i)}(x) + x \sum_{j = 1}^M h_{\lk_{\Delta_i}(e_i \cup a_j)}(x) - x \sum_{\ell = 1}^N h_{\lk_{\Delta_i}(e_i \cup b_\ell)}(x).  \]
		
		If we repeat the same thing with some edge $\alpha_i \in \Delta_i$ disjoint from $e_i$ and $q_i$ belonging to $\lk_{\Delta_i}(e_i \cup q_i)$, we can write \[ h_{\lk_{\Delta_i}(e_i \cup q_i \cup \alpha_i)}(x) + x P(x) \] for some polynomial $P(x)$ with nonnegative coefficients. \\

		In general, we end up with a polynomial of the form \[ (1 + x)^{2m} h_{\lk_\Delta(e_1 \cup \cdots \cup e_m)}(x) + \ldots + x^r (1 + x)^{2m - 2r} \left( \sum_{j = 1}^{M_m} h_{\lk_{\Delta_m}(\alpha_{1j} \cup \cdots \cup \alpha_{mj})}(x) - \sum_{\ell = 1}^{N_m} h_{\lk_{\Delta_m}(\beta_{1\ell} \cup \cdots \cup \beta_{m\ell})}(x) \right) + \ldots \] with $M_m \ge N_m$ after $m$ turns. \\
		
		A term in step $m$ of of the form $x^r (1 + x)^{2m - 2r}$ followed by a linear combination terms of the form $h_{\lk_{\Delta_m}(e_1 \cup \cdots \cup e_m)}(x)$ for disjoint edges $e_1, \ldots, e_m \in \Delta$ such that $e_1 \cup \cdots \cup e_m \in \Delta$ comes from decomposing the $h$-polynomials of the lines involved in step $m - 1$. In particular, the double suspension of the link of a given term over a fixed vertex involves multiplying a term of the form $x^r (1 + x)^{2m - 2 - 2r} \cdot \sum h_{\lk_{\Delta_{m - 1}}(e_1 \cup \cdots \cup e_{m - 1})}(x)$ by $(1 + x)^2$ and the ``remainder terms'' from a net nonnegative number of edge subdivisions (which have nonnegative differences from individual parts we're decomposing) are terms of the form $x^{r - 1} (1 + x)^{(2m - 2) - 2(r - 1)} \cdot P(x) = x^{r - 1} (1 + x)^{2m - 2r} \cdot P(x)$ multiplied by $x$. \\
		
		After using the decomposition \[ h_\Delta(x) = (1 + x)^2 h_{\lk_\Delta(e)}(x) + x h_{\lk_{\Delta_0}(e_0)}(x) + x h_{\lk_{\Delta_1}(e_1)}(x) + \ldots + x h_{\lk_{\Delta_r}(e_r)}(x) \] again for the individual terms from the $m^{\text{th}}$ turn, we have decompositions of terms induced by multiplication by $(1 + x)^2$ and $x$ respectively. Note that the next turn takes a link over another edge of $\Delta$ which is disjoint from the rest of the edges being considered. \\
		
	\end{proof}


	Alternatively, Proposition \ref{susplowerbd} implies that the existence of a Boolean decomposition of a simplicial complex $\Gamma_e$ such that $f(\Gamma_e) = h(\lk_\Delta(e))$ implies the same for some ``new'' simplicial complex $\Cone^2 \Gamma_e$ such that $f(\Cone^2 \Gamma_e) = h(\Susp^2 \lk_\Delta(e))$. This fills in the remaining Boolean parts which would be needed for a Boolean decomposition of a simplicial complex $\Gamma$ such that $f(\Gamma) = h(\Delta)$. Note that the $\lk_\Delta(e)$ cover $\Delta$. \\
	
	However, we need to check whether a vertex which is Boolean in one ``chart'' from $\lk_\Delta(e_1)$ for an edge $e_1 \in \Delta$ is still Boolean in a chart for $\lk_\Delta(e_2)$ from another edge $e_2 \in \Delta$. In particular, these non-Boolean vertices are from new vertices obtained from edge subdivisions moving from $\lk_\Delta(e_1)$ to $\lk_\Delta(e_2)$. Note that edge subdivisions don't pose a problem for Boolean decompositions since adding a new non-Boolean vertex in the $d$-dimensional setting reduces ``ambient set'' of Boolean coordinates to consider by 2. In addition, Proposition \ref{suspbool} and Proposition \ref{vertsubdivcount} imply that there is a net nonnegative number of edge subdivisions to get from $\Susp^2 \lk_\Delta(e_1)$ to $\Delta$ (which contains $\lk_\Delta(e_2)$). If a Boolean decomposition exists, it seems like there are $\le 2$ vertices which are ``locally Boolean'' but not ``globally Boolean''. To find when these would occur, we consider the degree 2 part $h_2(\Delta) = h_{d - 2}(\Delta)$ of the Dehn--Sommerville relations $h_k(\Delta) = h_{d - k}(\Delta)$. This means looking at relevant basis elements of $A^2(\Delta)$/edges of $\Gamma$. \\

	In this context, we look at possible overlaps between $\lk_\Delta(e_1)$ and $\lk_\Delta(e_2)$ (e.g. those of size $d - 2$). If $e_1 \cup e_2 \in \Delta$, we have that $\lk_\Delta(e_1) \cap \lk_\Delta(e_2) = \lk_\Delta(e_1 \cup e_2)$ since $\Delta$ is flag. Then, the largest faces in the intersection have size $d - 3$ or $d - 4$ if $e_1 \ne e_2$. In general, we have that $H \in \lk_\Delta(e_1) \cap \lk_\Delta(e_2)$ if and only if $e_1, e_2 \in \lk_\Delta(H)$. If $\dim \lk_\Delta(e_1) \cap \lk_\Delta(e_2) \ge d - 4$, then the number of coordinates which are ``locally Boolean'' but not ``globally Boolean'' is $\le 2$. The introduction of non-Boolean vertices is related to edges $e_2 \in \lk_{\Delta_1}(e_1)$ which are subdivided to form $\Delta_2$. \\

	\begin{rem} \textbf{(Repeated decompositions yielding gamma vectors, spanning $A^\cdot(\Delta)$ as a vector space, and Lefschetz-type maps) \\}
		
		We can make some further comments in the vein of Remark \ref{suspbool} on the structure of the decomposition from Theorem \ref{booldecgeo}. \\

		\begin{enumerate}
			
			\item \textbf{(Vertices underlying the $h$-polynomial induced by the decomposition) \\} 
			
			The role of the Boolean vertices $[d]$ is played by the collections of (disjoint) edges that we take links over while multiplying by $(1 + x)^2$ in the double suspension approximations. For example, this includes disjoint edges from a fixed facet of $\Delta$ containing $e \in \Delta$ in the approximation of $\Delta$ by $\Susp^2(\lk_\Delta(e))$ using an edge $e \in \Delta$. Going through the nested ``tree'' of operations of edge links inducing double suspension approximations and remainder terms, there is an induced basis of $A^\cdot(\Delta)$ as a vector space determined by (disjoint) unions of edges we take links over and adding subdivision vertices while removing parts spanned by vertices identified in edge contractions. \\

			\item \textbf{(Boolean decompositions and identifying facets) \\}
			
			Note that the same factor $x$ plays the role of different subdividing vertices added for each edge subdivision. Combining this with the observation above, a literal simplicial complex $\Gamma = \{ F \cup G : F \in S, G \in 2^{[d - 2|F|]} \}$ with the vertex set $[d]$ disjoint from $S$ such that $f(\Gamma) = h(\Delta)$ glues a collection of facets $\Delta$ (induced by edges we take links over in the double suspension steps) which are identified with a single fixed facet of $\Delta$ that contains $e$. Heuristically, this is similar to parametrizing basis elements of $\Delta$ without making a ``change of coordinates'' going from $\lk_\Delta(e)$ to $\lk_\Delta(e')$ for a different edge $e' \in \Delta$. It is natural to consider how we would make use of this assuming that a there is a simplicial complex $\Gamma_e$ with a Boolean decomposition $\Gamma_e = \{ F_e \cup G_e : F_e \in S_e, G_e \in 2^{[d - 2|F_e| - 2]} \}$ such that $f(\Gamma_e) = h(\lk_\Delta(e))$. Note that the $f$-vector side may be more flexible with such adjustments since algebraic results involving the $h$-vector often use the Cohen--Macaulay property. In addition, $\Gamma$ and $\Gamma_e$ can be taken to be balanced and connections to colorings are discussed in the construction filling in missing colors on p. 30 -- 32 of \cite{BFS}. \\

			\item \textbf{(Comments on Lefschetz-type maps) \\} 
			
			If we think about the Boolean decomposition proposed in Proposition 6.2 on p. 1377 of \cite{NPT} as a sort of Lefschetz decomposition, the ``Lefschetz-type'' maps yielding multiplication by $x^r$ in $x^r (1 + x)^{2m - 2r}$ at step $m$ seem to be connected to taking remainders $h_\Delta(x) - (1 + x)^2 h_{\lk_\Delta(e)}(x)$ of double suspension approximations. However, this decreases the dimension by 2 instead of increasing the dimension of the flag sphere considered by 2. A sort of inverse map would be something like a double suspension of a flag sphere followed by a net nonnegative number of edge subdivisions (i.e. $M$ edge subdivisions and $N$ edge contractions with $M \ge N$) leading to a choice of a flag sphere with dimension 2 higher than the starting one that contains the starting flag sphere as an edge link. Also, the basis elements showing up as images of such maps (under some degree) seem to include any terms which aren't from the boundary of the cross polytope of the given dimension in $\Susp^2(\lk_\Delta(e))$ and anything in the remainder $h_\Delta(x) - (1 + x)^2 h_{\lk_\Delta(e)}(x)$. \\

		\end{enumerate}
	\end{rem}
	
	\color{black} 
	
	In addition, the explicit decomposition above implies the following: \\
	
	\begin{cor} \label{locglobgeo}
		Suppose for every edge $e' \in \Delta$ that there is a simplicial complex $\Gamma_{e'}$ such that $f(\Gamma_{e'}) = h(\lk_\Delta(e'))$ and $\Gamma_{e'}$ has a Boolean decomposition. \\
		
		\begin{enumerate}
			\item If $\Gamma_e$ has a Boolean decomposition with the cone points of $\Cone^2 \Gamma_e$ contributing to the ``global'' Boolean vertices, we can ``convert'' the one from $\Gamma_{e'}$ for the unique facet of $\Gamma_{e'}$ to ``global'' coordinates after changing at most 2 of the Boolean vertices to non-Boolean ones. \\
			
			\item If the number of vertices of $\Gamma_{e'}$ changing from locally Boolean to globally non-Boolean is at most 1 for each edge $e' \in \Delta$, then the expansion of the number of Boolean coordinates by 2 via the double suspension on the $\Delta$ side (equivalently double coning on the $\Gamma$ side) means the Boolean decomposition would expand to all of $\Gamma$. \\
		\end{enumerate}
	\end{cor}

	
	\color{black}

\end{document}